\documentclass[11pt]{amsart}
\usepackage{graphicx,amsfonts,amssymb,amsmath,amsthm,url,
  verbatim,amscd}
  \usepackage{color}
\usepackage[usenames,dvipsnames]{xcolor}
\usepackage[normalem]{ulem}
\usepackage{pdfsync}

\theoremstyle{plain} %
\newtheorem{theorem}  {Theorem}    [section]
\newtheorem{lemma}      [theorem]{Lemma}
\newtheorem{corollary}  [theorem]{Corollary}
\newtheorem{proposition}[theorem]{Proposition}

\newtheorem{conjecture} {Conjecture}
\newtheorem{question}{Question}

\theoremstyle{definition}
\newtheorem{definition} [theorem]{Definition}

\theoremstyle{remark}
\newtheorem{remark}  [theorem]            {Remark}

\allowdisplaybreaks

\renewcommand{\H}{\mathbb H}
\newcommand{\A}{{\mathbb A}}
\renewcommand{\a}{{\mathbf a}}
\newcommand{\f}{{\mathbf f}}
\newcommand{\Q}{{\mathbb Q}}
\newcommand{\Z}{{\mathbb Z}}

\newcommand{\R}{{\mathbb R}}
\newcommand{\C}{{\mathbb C}}
\newcommand{\bs}{\backslash}

\newcommand{\p}{\mathfrak p}
\newcommand{\OF}{{\mathfrak o}}
\newcommand{\GL}{{\rm GL}}

\newcommand{\sgn}{{\rm sgn}}

\newcommand{\mat}[4]{{\setlength{\arraycolsep}{0.5mm}\left[
      \begin{array}{cc}#1&#2\\#3&#4\end{array}\right]}}

\def\twist{{}}
\def\subscriptp{{}}

\def\vol{\operatorname{vol}}
\def\GL{\operatorname{GL}}

\def\eps{\varepsilon}

\begin{document}

\bibliographystyle{plain}

\title{Large values of newforms on $\GL(2)$ with highly ramified central character}

\author{Abhishek Saha}
\address{Department of Mathematics\\
  University of Bristol\\
  Bristol BS81TW\\
  UK} \email{abhishek.saha@bris.ac.uk}

\thanks{The author is partially supported by EPSRC grant EP/L025515/1.}

\begin{abstract}
We give a lower bound for the sup-norm of an $L^2$-normalized newform in  an irreducible, unitary, cuspidal representation $\pi$ of $\GL_2$ over a number field.  When the central character of $\pi$ is sufficiently ramified, this bound improves upon the trivial bound by a positive power of $N$ where $N$ is the norm of the conductor of $\pi$.  This generalizes a result of Templier, who dealt with the special case when the conductor of the central character equals the conductor of the representation. We also make a conjecture about the true size of the sup-norm in the $N$-aspect that takes into account this central character phenomenon. Our results depend upon some explicit formulas and bounds for the Whittaker newvector over a non-archimedean local field, which may be of independent interest.
\end{abstract}
\maketitle

\def\eachnotany{~{each}~}

\section{Introduction}Let $f$ be either a holomorphic cuspform of weight $k \ge 1$ or a Maass cuspform of weight $k\in \{0,1\}$ and eigenvalue $\lambda$, with respect to the subgroup $\Gamma_1(N)$. Further, assume that $f$ is a newform. Let $\chi$ denote the character of $f$, and $M$ denote the conductor of $\chi$; in particular $M$ divides $N$ and $\chi(-1) = (-1)^k$. The function $F(z)=y^{k/2}|f(z)|$ is a non-negative real-valued $\Gamma_0(N)$-invariant function on the upper-half plane that vanishes at the cusps, and so it is natural to define $$\|f\|_\infty = \sup_{z \in \Gamma_0(N) \bs \H} F(z), \qquad \|f\|_2 = \left (\vol(\Gamma_0(N) \bs \H)^{-1}\int_{\Gamma_0(N) \bs \H} F(z)^2 dz \right)^{1/2}.$$

The problem of bounding $\frac{\|f\|_\infty}{\|f\|_2}$ as the parameter $N$ varies is interesting from various  points of view and has been the topic of several recent works. The ``trivial bound"  is $$N^{-\eps} \ll_{\lambda/k,\eps} \frac{\|f\|_\infty}{\|f\|_2} \ll_{\lambda/k,\eps}N^{1/2 + \eps}.$$  The first non-trivial upper bound was obtained by Blomer and Holowinsky~\cite{blomer-holowinsky} in 2010, who proved $\|f\|_\infty \ll_{\lambda/k, \epsilon} N^{\frac{216}{457} + \epsilon}$ for squarefree $N$. Since then, there has been several improvements. The following are the best currently available upper bounds:

\begin{itemize}

  \item $\frac{\|f\|_\infty}{\|f\|_2}\ll_{\lambda/ k, \epsilon} N^{\frac{1}{3} + \epsilon}$ for squarefree $N$ due to Harcos and Templier~\cite{harcos-templier-2}.

  \item     $\frac{\|f\|_\infty}{\|f\|_2}\ll_{\lambda/ k, \epsilon} N^{\frac{5}{12} + \epsilon}$ for any $N$, and $M=1$, due to the author~\cite{sahasuplevel}.
\end{itemize}

\begin{remark}The paper of Harcos and Templier~\cite{harcos-templier-2} assumes that $\chi$ is the trivial character (i.e., $M=1$) but its methods can be extended to cover the case of general $\chi$.\end{remark}

\begin{remark}One could also ask a related question that focuses only on the \emph{bulk} and not the \emph{cusps}. In that direction, Marshall~\cite{marsh15} has recently proved that whenever $M=1$, $\frac{\|f'|_{\Omega}\|_\infty}{\|f'\|_2}\ll_{\lambda/ k, \Omega, \epsilon} N_1^{1/2 + \epsilon}$ where $\Omega$ is any compact set, $f'$ is a certain shift of $f$, and where $N_1$ is the smallest integer with $N|N_1^2$.

\end{remark}

This begs the question: what is the true size of  $\frac{\|f\|_\infty}{\|f\|_2}$? Applying some heuristics --- mean value estimates, the Lindel\"of hypothesis, the case of oldforms --- seemed to suggest that the following optimal bound might be always true:

\begin{equation}\label{e:folklore} \frac{\|f\|_\infty}{\|f\|_2} \ll_{\lambda/k, \eps} N^{\eps} \end{equation}

In~\cite{templier-large}, Templier referred to~\eqref{e:folklore} as a ``folklore conjecture". Nonetheless, in the same paper, he showed that the conjecture as stated is false. Indeed, he was able to provide the following wide class of counterexamples.

\begin{theorem}[Templier]
Let $f$, $M$, $N$ be as above and suppose that $M=N$. Then $$\frac{\|f\|_\infty}{\|f\|_2} \gg_{\lambda/k, \eps} N^{-\eps} \prod_{p^c \| N} p^{\frac12 \lfloor \frac{c}{2} \rfloor}.$$
In particular, if $N$ is a perfect square then  $$\frac{\|f\|_\infty}{\|f\|_2} \gg_{\lambda/k, \eps} N^{1/4 - \eps}$$ and hence the bound \eqref{e:folklore} is strongly violated by such $f$.
\end{theorem}

 Templier's result  tells us that the newforms with powerful level and maximally ramified central character fail the bound \eqref{e:folklore} spectacularly. However, several follow-up questions naturally present themselves. Are Templier's class of examples the only provable counterexamples to the bound \eqref{e:folklore}? What happens when the character is not maximally ramified? How should we modify the ``folklore conjecture"? In this paper we try to answer some of these questions.

The main theorem of this paper (Theorem~\ref{t:globalmain}) is an extension of Templier's result to the case of other central characters. It gives a lower bound  for $\frac{\|\phi\|_\infty}{\|\phi\|_2}$ where $\phi$ is  a newform on $\GL_2(\A_F)$ and $F$ is any number field. This bound involves the conductor of $\phi$, the conductor of its central character, and the archimedean type. In the special case $F=\Q$, and suppressing the archimedean dependence, our result can be stated as follows.
 \medskip

 \textbf{Theorem A.} (cf.  Theorem~\ref{t:globalmain}) \emph{ Let $f$, $M$, $N$ be as above, with prime decompositions $M = \prod_p p^{m_p}$, $N = \prod_p p^{n_p}$. Then \begin{equation}\label{e:globalmain}\frac{\|f\|_\infty}{\|f\|_2} \gg_{\lambda/k, \eps} N^{-\eps}\prod_{p}\max(p^{\frac12 \lfloor \frac{3m_p}{2} \rfloor - \frac{n_p}2}, 1).\end{equation}
In particular, if $M$ is a perfect square and $N^2$ divides $M^3$, then  $$\frac{\|f\|_\infty}{\|f\|_2} \gg_{\lambda/k, \eps} N^{-\eps} \frac{M^{3/4}}{N^{1/2}}$$ and hence the bound \eqref{e:folklore} is violated by such $f$ whenever $M > N^{2/3 + \delta}$ for any fixed $\delta>0$.}

\medskip

Theorem A has some interesting features. Note that the product on the right side of \eqref{e:globalmain} is larger than 1 if and only if there is a prime $p$ such that  $m_p \ge \frac23(n_p+1)$. In other words, we can find counterexamples to the ``folklore conjecture" even when $M \neq N$, provided that $M$ is not too small (in an arithmetic sense) compared to $N$. Furthermore, as $M$ gets smaller, so does the ratio $\frac{\|f\|_\infty}{\|f\|_2}$. The situation is best illustrated by looking at the case of prime power levels $M=p^m$, $N=p^n$ where $p$ is some fixed prime. Then, as $m$ varies (relative to $n$), we get a ``phase transition" at $m\asymp \frac{2n}3$. It is only above this range that our theorem gives counterexamples to the bound \eqref{e:folklore}. This is somewhat curious, and may remind the reader of the transition range of the Bessel function.

In the case that the central character is ``not too highly ramified", \emph{we conjecture that the bound~\eqref{e:folklore} is true. } Precisely:
\begin{conjecture}Let $f$, $M$, $N$ be as above and let $N_1$ be the smallest integer with $N|N_1^2$. Then whenever $M$ divides $N_1$, we have

\begin{equation}
\label{newfolk}
\frac{\|f\|_\infty}{\|f\|_2} \ll_{\lambda/k, \eps} N^{\eps}.
\end{equation}\end{conjecture}

This conjecture would follow if one assumes that the only obstructions to $\frac{\|f\|_\infty}{\|f\|_2}$ being as small as possible are local, and if Conjecture~\ref{newconj} below is true.

Our method to prove Theorem A (or the more general Theorem~\ref{t:globalmain}) uses the fact that $$\frac{\|f\|_\infty}{\|f\|_2} \gg_{\lambda/k, \eps} N^{-\eps} \prod_v h(\pi_v),$$ where $h(\pi_v)=\max|W_{\pi_v}|\underline{}$ is the corresponding quantity for the newform in the \emph{local Whittaker model}. This was also Templier's strategy, and indeed a similar idea has then been used (in the archimedean aspect) by Brumley and Templier~\cite{brumley-templier} to give lower bounds for sizes of $\GL(n)$ newforms. More precisely, we prove the following purely local result, which may be of independent interest.

\medskip

\textbf{Theorem B.} (Theorem~\ref{t:main}.) \emph{Let $F$ be a non-archimedean local field of characteristic 0 and residue characteristic $q$. Let $\pi$ be a generic
irreducible admissible unitarizable
representation of $\GL_2(F)$ such that the conductor of $\pi$ equals $q^n$ and the conductor of $\omega_\pi$ equals $q^m$. Let $W_{\pi}$ be the Whittaker newform for $\pi$ normalized so that $W_{\pi}(1)=1$. Then we have \begin{equation}
\label{localresult}\max(q^{\frac12 \lfloor \frac{3m}{2} \rfloor - \frac{n}2}, 1)  \ll \max_{g \in \GL_2(F)} |W_\pi(g)| \ll q^{\frac12 \lfloor \frac{n}{2} \rfloor}.\end{equation}}
\medskip

Before saying a few words about the proof of Theorem B, we would like to pose a question.

\begin{question}\label{que}Let the setup be as in Theorem B. What is the true size of $\max_{g \in \GL_2(F)} |W_\pi(g)| $?

\end{question}

In upcoming work, we hope to answer Question~\ref{que} for all $\pi$. However, for now, we would like to make the following conjecture, which may be viewed as a local Lindel\"of hypothesis for Whittaker newforms.

\begin{conjecture}\label{newconj}Let $F$ be a non-archimedean local field of characteristic 0 and residue characteristic $q$. Let $\pi$ be a generic
irreducible admissible unitarizable
representation of $\GL_2(F)$ such that the conductor of $\pi$ equals $q^n$ and the conductor of $\omega_\pi$ equals $q^m$. Suppose that $m \le \lceil \frac{n}{2} \rceil$. Let $W_{\pi}$ be the Whittaker newform for $\pi$ normalized so that $W_{\pi}(1)=1$. Then we have \begin{equation}
\label{localresult} 1  \ll \max_{g \in \GL_2(F)} |W_\pi(g)| \ll_{\epsilon} q^{n \epsilon}.\end{equation}
\end{conjecture}

As for the proof of Theorem B, the main tool is the local functional equation of Jacquet-Langlands. Using this, we are able to get an explicit formula (Proposition~\ref{basicformula}, which we call ``the basic identity") for the Whittaker newform for all generic representations. This method  has the advantage that it does not require the existence of an induced model and so our formula is also valid for supercuspidals. However, in the special case when $m>n/2$ (in the notation of Theorem B), the representation $\pi$ must be a principal series representation, and in this case our formula simplifies into a sum of products of $\GL(1)$ $\varepsilon$-factors. Furthermore, in the range $m > 2n/3$, we can prove that at certain special values of $g$, the additive and multiplicative characters appearing in the expression for $W_\pi(g)$ conspire to produce large values.

We note that our method differs slightly from the one Templier used in~\cite{templier-large} to deal with his case.\footnote{However, we have been recently informed by Templier that he was aware of the functional equation method and in his manuscript \cite{temp:notes} under preparation has used it to double-check his earlier results.} Templier used an intertwining operator from the induced model to the Whittaker model to get a formula for the Whittaker newform. In contrast, we use the local functional equation to get the formula for the Whittaker newform directly.  It is likely that Templier's method can also be used to obtain our result. We also prove many other exact identities for the local Whittaker newform which should be of independent interest. In some forthcoming papers, we will use the local machinery developed in this paper to pursue diverse applications, ranging from improved global sup-norm bounds (see our forthcoming paper \cite{saha-sup-level-hybrid}) to a careful study of the ramification index at cusps for the modular parametrization. Finally, it would be extremely interesting to consider analogues of Theorem A and Theorem B for other groups. We intend to return to this question in the future.

\subsection*{Acknowledgements}I would like to thank the following people: Andy Booker for providing me with the formulas for the archimedean Whittaker function in the case of $\GL_2(\C)$, Paul Nelson for several helpful discussions, Ameya Pitale for carefully reading an earlier version of this manuscript, and Ralf Schmidt for teaching me how to exploit the local functional equation to study local Whittaker functions and for some useful feedback. I would also like to thank Yueke Hu for noticing that the conjecture stated in an earlier version of this paper was not correct.

After an earlier draft of this paper was made available, Nicolas Templier kindly communicated to me his (under preparation) manuscript \cite{temp:notes} on this topic where he also obtains several of the results proved in this paper.

\section{Bounds for the local Whittaker newform}
\label{sec:local-calculations}
\subsection{Some notations}\label{sec:2-notations}

Let  $F$ be a non-archimedean local
  field  of characteristic zero whose
  residue field has cardinality $q$.
Let $\mathfrak{o}$ be its ring of integers,
and $\mathfrak{p}$ its maximal ideal.
Fix a generator $\varpi$ of $\mathfrak{p}$.
Let $|.|$ denote the absolute value
on $F$ normalized so that
$|\varpi| = q^{-1}$. For each $x \in F^\times$, let $v(x)$ denote the integer such that $|x| = q^{-v(x)}$. Define $\zeta_F(s) = (1-q^{-s})^{-1}.$
% if
% $a = \varpi^n u$ with $u \in \mathfrak{o}^\times$,
% then $|a| = p^{-n}$.

Let $G = \GL_2(F)$ and $K = \GL_2(\mathfrak{o})$.
For each integral ideal $\mathfrak{a}$ of $\mathfrak{o}$,
let
\[
K_0(\mathfrak{a}) = K \cap \begin{bmatrix}
  \mathfrak{o}  & \mathfrak{o}  \\
  \mathfrak{a}  & \mathfrak{o}
\end{bmatrix}, \quad
K_1(\mathfrak{a})
= K \cap \begin{bmatrix}
  1+ \mathfrak{a}  & \mathfrak{o}  \\
  \mathfrak{a}  & \mathfrak{o}
\end{bmatrix}, \quad
K_2(\mathfrak{a})
= K \cap \begin{bmatrix}
\mathfrak{o}    & \mathfrak{o}  \\
  \mathfrak{a}  & 1+ \mathfrak{a}
\end{bmatrix}.
\]

In particular, $K_0(\mathfrak{o}) = K_1(\mathfrak{o}) =  K_2(\mathfrak{o}) = K$.   Also, we note that for each non-negative integer $n$, $$K_2(\mathfrak{\p^n})
=  \begin{bmatrix}0 &1\\ \varpi^n & 0\end{bmatrix}K_1(\mathfrak{\p^n}) \begin{bmatrix}0 &\varpi^{-n}\\ 1 & 0\end{bmatrix}.
$$
Write
\[
w = \begin{bmatrix}
  0 & 1 \\
  -1 & 0
\end{bmatrix},
\quad
a(y) = \begin{bmatrix}
  y &  \\
  & 1
\end{bmatrix},
\quad
n(x) = \begin{bmatrix}
  1 & x \\
  & 1
\end{bmatrix},
\quad z(t)
= \begin{bmatrix}
  t &  \\
  & t
\end{bmatrix}
\]
for $x \in F, \ y \in F^\times, \ t \in F^\times$.
Define subgroups
$N =
\{n(x):  x\in F \}$,
$A = \{a(y): y\in F^\times \}$,
$Z =\{ z(t):
t \in F^\times \}$,
and $B = Z N A = G \cap
\left[
  \begin{smallmatrix}
    *&*\\
    &*
  \end{smallmatrix}
\right]$ of $G$.

We normalize Haar measures as
follows.
The measure $dx$ on the additive group $F$ assigns volume 1
to $\OF$, and transports to a measure on $N$.
The measure $d^\times y$ on the multiplicative group $F^\times$ assigns
volume 1 to $\OF^\times$,
and transports to measures on
$A$ and $Z$.
We obtain a left Haar measure $d_Lb$ on $B$ via
$$d_L(z(u)n(x)a(y)) = |y|^{-1}\, d^\times u \, d x \, d^\times
y.$$
Let $dk$ be the probability Haar measure on $K$.
The Iwasawa decomposition
$G = B K$ gives a left Haar measure $dg = d_L b \, d k$ on $G$.

For each character\footnote{We adopt the convention that
  a \emph{character} of a topological group is a continuous (but
  not necessarily unitary) homomorphism into $\C^\times$.}
$\sigma$ of $F^\times$,
there exists a minimal non-negative integer $a(\sigma)$ such that
$\sigma(1+t)=1$  for all $t \in \p^{a(\sigma)}$.
For each irreducible admissible
representation $\sigma$ of $G$, there exists a minimal non-negative integer $a(\sigma)$ such that $\sigma$ has a $K_1(\p^{a(\sigma)})$-fixed vector. In either case, the integer  $q^{a(\sigma)}$ is called the
local analytic conductor\footnote{In the rest of this section, we
  will often drop the words ``local analytic" for brevity and
  call this simply the ``conductor".} of $\sigma$;
we denote it by
$C(\sigma)$.  We let $\omega_\sigma$ denote the central character of $\sigma.$
We let $\tilde{X}$ denote the group of characters $\chi$ of $F^\times$ such that $\chi(\varpi)=1$. Thus $\tilde{X}$ is isomorphic to the group of characters of $\OF^\times$.

Fix an additive character $\psi : F \rightarrow \mathbb{C}^1$
with conductor $\mathfrak{o}$. For $\sigma$ a character of $F^\times$ or an irreducible admissible representation of $G$, we let $L(s, \sigma)$ denote the local $L$-factor and $\eps(s, \sigma)=\eps(s, \sigma, \psi)$ denote the local $\eps$-factor; these factors are defined in~\cite{MR0401654}. Some properties of the  $\eps$-factor we will need are:

 \begin{enumerate}
 \item  $\eps(s, \sigma) =  \eps(1/2, \sigma)q^{-a(\sigma)(s-1/2)}.$

 \item  $\eps(1/2, \sigma)\eps(1/2, \tilde{\sigma}) = \omega_\sigma(-1),$ where     $\tilde{\sigma}$ is the contragredient of $\sigma$.

\item   $|\eps(1/2, \sigma)| = 1$ whenever $\sigma$ is unitary.

 \end{enumerate}

 For each  $\mu \in \tilde{X}$, and each $x \in F$, define the Gauss sum $$G(x, \mu) = \int_{\OF^\times}\psi(xy) \mu(y) d^\times y.$$
It is a well-known fact that $G(x, \mu) = 0$ unless $\mu$=1 or $v(x) = -a(\mu)$. More precisely, we have the following formula:

\begin{equation}\label{formulagauussum}G(x, \mu) = \begin{cases}1 & \text{ if } \mu=1, v(x) \ge 0, \\ -\zeta_F(1)q^{-1} &\text{ if } \mu=1, v(x) =-1, \\ 0 & \text{ if } \mu=1, v(x)<-1,\\ \zeta_F(1)|x|^{-1/2}\eps(1/2, \mu^{-1})\mu^{-1}(x)  &\text{ if } \mu\neq 1, v(x)=-a(\mu), \\ 0  &\text{ if } \mu\neq 1, v(x)\neq -a(\mu) \end{cases}\end{equation}

 For each generic representation $\sigma$
of $G$, let $\mathcal{W}(\sigma\subscriptp, \psi)$
denote the Whittaker model of $\sigma\subscriptp$ with respect
to $\psi$ (see~\cite{MR0401654}).
For two characters $\chi_1$, $\chi_2$ on $F^\times$, let $\chi_1
\boxplus \chi_2$ denote the principal series representation on
$G$
that is unitarily induced from the corresponding representation of  $B$. The usual induced model for  $\chi_1
\boxplus \chi_2$ consists of smooth functions $f$ on $G$ satisfying $$f\left(\mat{a}{b}{0}{d} g\right) = |a/d|^{\frac12} \chi_1(a) \chi_2(d)  f(g),$$ but we will be working purely in the Whittaker model $\mathcal{W}(\chi_1
\boxplus \chi_2, \psi)$.

For each integer $k$, we let $\tilde{X}(k)$ denote the set of $\chi \in \tilde{X}$ such that $a(\chi)\le k$. For an irreducible admissible representation $\sigma$ of $G$ and a character $\chi$
of $F^\times$, write
$\sigma \chi$ for the representation
$\sigma \otimes (\chi \circ \det)$ of $G$.

We use the notation
$A \ll_{x,y,z} B$
to signify that there exists
a positive constant $C$, depending at most upon $x,y,z$,
so that
$|A| \leq C |B|$. The symbol $A \ll B$ means the same thing except that this positive constant is now absolute.  The symbol $\epsilon$ will denote a small positive quantity, whose value may change from line to line. The notation $A\asymp B$ means that $B \ll A \ll B$.
\subsection{The local Whittaker newform}
\emph{For the rest of Section~\ref{sec:local-calculations}, we let $\pi$ denote a generic
irreducible admissible unitarizable
representation of $G$ and put $n = a(\pi)$.}
If $n=0$, then $\pi$ is spherical, i.e., $\pi = \chi_1 \boxplus \chi_2$ with $\chi_1$ and $\chi_2$ being unramified characters of $F^\times$. This case is well-understood and so we will restrict ourself to $\pi$ with $n\ge 1$. Any such $\pi$ is one of the following types:

\begin{enumerate}

\item $\pi = \chi \rm{St}$, a twist of the Steinberg representation with an unramified unitary character $\chi$. These have $n=1$.

\item $\pi = \chi_1 \boxplus \chi_2$ where $\chi_1$, $\chi_2$ are unitary characters with $a(\chi_1)>0$ and $a(\chi_2)=0$. These have $n = a(\chi_1)$.

 \item $\pi$ with $L(s, \pi)=1$. These  consist of    the following three subcases:\textbf{ a)} $\pi = \chi \rm{St}$ with $\chi$ unitary ramified; these have $n=2a(\chi)\ge 2$, \textbf{b)} $\pi = \chi_1 \boxplus \chi_2$ where  $\chi_1$, $\chi_2$ are unitary characters with $a(\chi_1) \ge a(\chi_2)>0$; these have $n=a(\chi_1)+a(\chi_2) \ge 2$, \textbf{ c)} $\pi$ supercuspidal; these also have $n \ge 2$.

\end{enumerate}

\begin{definition}\label{defn:normalized-W-newform}
  The \emph{normalized Whittaker newform}
  $W_\pi $
  is
  the unique vector in $\mathcal{W}(\pi,\psi)$ invariant under $K_1(\p^{n})$
  that satisfies $W_\pi(1)= 1$.
\end{definition}

\begin{definition}\label{defn:normalized-W-newform}
  The \emph{normalized Whittaker conjugate-newform}
  $W_\pi^* $
  is
  the unique vector in $\mathcal{W}(\pi,\psi)$ invariant under $K_2(\p^{n})$
  that satisfies $W_\pi^*(1)= 1$.
\end{definition}

If the central character of $\omega_\pi$ is unramified, then $W_\pi$ equals $W_\pi^*$ (and is right-invariant under the larger group $K_0(\p^n)$). However in general, these vectors are different, and it is necessary to use the ``generalized Atkin-Lehner relation" between these vectors (see Corollary~\ref{prevtil2} and Proposition~\ref{cpiformula}) to obtain optimal bounds concerning either.

\begin{remark} We will soon see that $W_{\tilde{\pi}}(g) =  \omega_{\pi}^{-1}(\det(g))W_\pi^*(g)$ where $\tilde{\pi}= \omega_\pi^{-1}\pi$ is the contragredient representation of $\pi$.\footnote{This can also be seen directly, using the fact that the right side represents a vector in  $\mathcal{W}(\tilde{\pi},\psi)$ invariant under $K_1(\p^n)$.}
\end{remark}

\begin{remark} If $\chi$ is an unramified character, then $W_{\pi\chi} (g) = \chi(\det(g)) W_\pi(g)$ (resp., $W^*_{\pi\chi} (g) = \chi(\det(g)) W^*_\pi(g)$). So, in our study of the Whittaker newform, we can twist $\pi$ by an unramified character if doing so will bring it to a simpler form. For example, we may assume that $\omega_\pi \in \tilde{X}$ without affecting the generality of our results.
\end{remark}

The values of $W_\pi$ and $W_\pi^*$ on the diagonal $A$ are well-known; we record them below (see ~\cite{Sch02} for a proof).

\begin{lemma}\label{whitdiagonal}Let $n \ge 1$. Let $t \in \Z$, $v  \in \OF^\times$. We have \begin{equation}W_\pi(a(\varpi^t v)) =  \begin{cases}(\chi(\varpi)q^{-1})^t  & \text{ if } t\ge 0 \text{ and } \pi = \chi \rm{St},  \ \chi \text{ unramified, } \\ 0  & \text{ if } t< 0 \text{ and } \pi = \chi \rm{St},  \ \chi \text{ unramified, }\\ \chi_1(\varpi^t v)q^{-t/2} & \text{ if } t\ge 0 \text{ and } \pi = \chi_1 \boxplus \chi_2, \ a(\chi_1)>0, a(\chi_2)=0,\\ 0& \text{ if } t< 0 \text{ and } \pi = \chi_1 \boxplus \chi_2, \ a(\chi_1)>0, a(\chi_2)=0,\\
\omega_\pi( v) & \text{ if } t=0 \text{ and } L(s, \pi)=1, \\ 0 & \text{ if } t\neq 0 \text{ and } L(s, \pi)=1.
\end{cases}
\end{equation}
\begin{equation}W_\pi^*(a(\varpi^t v)) =  \begin{cases}(\chi(\varpi)q^{-1})^t  & \text{ if } t\ge 0 \text{ and } \pi = \chi \rm{St},  \ \chi \text{ unramified, } \\ 0  & \text{ if } t< 0 \text{ and } \pi = \chi \rm{St},  \ \chi \text{ unramified, }\\ \chi_2(\varpi^t )q^{-t/2} & \text{ if } t\ge 0 \text{ and } \pi = \chi_1 \boxplus \chi_2, \ a(\chi_1)>0, a(\chi_2)=0,\\ 0& \text{ if } t< 0 \text{ and } \pi = \chi_1 \boxplus \chi_2, \ a(\chi_1)>0, a(\chi_2)=0,\\
1 & \text{ if } t=0 \text{ and } L(s, \pi)=1, \\ 0 & \text{ if } t\neq 0 \text{ and } L(s, \pi)=1.
\end{cases}
\end{equation}
\end{lemma}

\subsection{The main result}
\begin{definition}We define $h(\pi) = \mathrm{sup}_{g \in G} |W_\pi(g)|$.

\end{definition}
Note that $h(\pi) \ge 1$, since  $W_\pi(1)= 1$.

\begin{lemma}\label{trivboundn0}If $n \le 1$, then $h(\pi) \le \frac32$. Furthermore, if $n \le 1$ and $q \ge 4$ then $h(\pi)=1$.

\end{lemma}
\begin{proof}The condition $n\le 1$ implies that $\pi$ is either equal to $\chi_1\boxplus \chi_2$ with $a(\chi_1) + a(\chi_2) \le 1$ or $\chi$ is an unramified twist of the Steinberg representation. In both cases, the lemma follows from well-known explicit formulas for the Whittaker newform; see e.g.~\cite{Sch02}.
\end{proof}

Our main result is the following.

\begin{theorem}\label{t:main}Let $\pi$ be a generic
irreducible admissible unitarizable
representation of $G$ and let $n = a(\pi)$, $m=a(\omega_\pi)$. Then we have $$\max(q^{\frac12 \lfloor \frac{3m}{2} \rfloor - \frac{n}2}, 1)  \ll h(\pi) \ll q^{\frac12 \lfloor \frac{n}{2} \rfloor}.$$
\end{theorem}

We will prove this Theorem over the following few subsections. Our proof shows that the implied constants can in fact be taken to be equal to $\frac32$.

\begin{remark}In the very special case $m=n$, our theorem reduces to  $h(\pi) \asymp q^{\frac12\lfloor \frac{n}{2} \rfloor}$. This fact was also proved by Templier~\cite{templier-large}, who used a somewhat different method than what we will use. Apart from Templier's result, there does not appear to have been any other previous work on the size of $h(\pi)$.
\end{remark}

\begin{remark}\label{rem:localcon}As mentioned in the introduction (see Conjecture~\ref{newconj}), we suspect that the bound on the left side of Theorem \ref{t:main} is actually the true order of magnitude of $h(\pi)$ up to $q^{n \epsilon}$, \emph{provided $m \le \lceil \frac{n}{2} \rceil.$} That is, we conjecture that \begin{equation}
\label{localresult} 1  \ll \max_{g \in \GL_2(F)} |W_\pi(g)| \ll_{\epsilon} q^{n \epsilon}, \qquad \text{whenever } m \le \lceil \frac{n}{2} \rceil.\end{equation}  For supercuspidal representations, this conjecture would follow if one knew square-root cancelation in the sum~\eqref{supercuspsum} below for the values of the Whittaker newform. In the range $m >\lceil \frac{n}{2} \rceil,$ we do not speculate about the true size of $h(\pi)$ at this point\footnote{In an earlier draft of this paper, we did speculate about the size of $h(\pi)$ in this range as well, but it was pointed out by Yueke Hu, and confirmed by our calculations, that those speculations were incorrect.} except to say that we believe that the lower bound in Theorem~\ref{t:main} can be somewhat improved. We will return to these questions, which we believe to be of independent interest, in future work.
\end{remark}

\begin{remark}\label{remchoice} Suppose that instead of the unramified character $\psi$ fixed above, we were to take a different character $\psi'$ given by $\psi'(x) = \psi(yx)$ for some $y \in F^\times$. Then there is a natural intertwining operator from $\mathcal{W}(\pi, \psi)$ to $\mathcal{W}(\pi, \psi')$ given by $W \mapsto W'$ where $W'(g) :=  W(a(y)g)$. This map takes newforms to newforms, i.e., $W_\pi'$ is the $K_1(\p^n)$-fixed vector in the model $\mathcal{W}(\pi, \psi')$. In particular as $h(\pi) = \sup_g|W_\pi(g)| =  \sup_g|W_\pi'(g)|$, we see that $h(\pi)$ does not depend on the choice of additive character $\psi$. \emph{Indeed, in our global application later, we will take a global additive character $\psi$ that may not have local conductor $\OF$ at all places.} The above remarks ensure that the results for $h(\pi)$ proved in this section would still be valid, despite the fact that they are only performed for a specific unramified choice of $\psi$.
\end{remark}

\subsection{An explicit set of representatives}

We now set about proving Theorem~\ref{t:main}. Since the result is trivially true for $n=0$, \emph{we henceforth assume that $n \ge 1.$}
\begin{definition}For each $t \in \Z$, $k \in \Z$ and $v \in \OF^\times$, we define $g_{t,k,v} = a(\varpi^t)wn(\varpi^{-k}v).$
\end{definition}

\begin{lemma}\label{disjoint}Put $k_n = \min(k,n-k)$. We have the disjoint union $$G = \bigsqcup_{t \in \Z}\bigsqcup_{0\le k \le n}\bigsqcup_{v \in \OF^\times/(1+ \p^{k_n})} Z N g_{t,k,v} K_1(\p^{n})= \bigsqcup_{t \in \Z}\bigsqcup_{0\le k \le n}\bigsqcup_{v \in \OF^\times/(1+ \p^{k_n})} Z N g_{t,k,v} K_2(\p^{n}).$$

\end{lemma}
\begin{proof}That any element of $G$ can be expressed in the given form follows from the Bruhat decomposition for $G$. The disjointness of the decomposition is a routine check based on comparing the matrix entries from different subsets. We omit the details, as they are elementary.
\end{proof}
Because of the above Lemma, it suffices to understand the values of  $W_\pi$ and $W_\pi^*$ at the matrices $g_{t,k,v}$ with $t \in \Z$, $0\le k \le n$, $v \in \OF^\times/(1+ \p^{\min(k,n-k)})$. In particular $h(\pi) = \max_{t,k,v}|W_\pi(g_{t,k,v})|$. We will soon see that we may further restrict $k$ to the range $0 \le k \le n/2$.

\begin{remark}\label{cpicompute1}The identity element of $G$ corresponds to $t=-2n$, $k=n$, $v=1$. Indeed we have the decomposition $$1 = z(-\varpi^{n})n(\varpi^{-n})g_{-2n,n,1}k$$ with $k \equiv 1 \pmod{\varpi^n}.$ This shows in particular that $$W_\pi(g_{-2n,n,1})=W_\pi^*(g_{-2n,n,1})=\omega_\pi(-\varpi^{-n})\psi(-\varpi^{-n}).$$

\end{remark}

\begin{definition}Let $\langle \ , \rangle$ be the (unique) invariant inner product on $\pi$ normalized in the Whittaker model as follows. For $W_1, W_2 \in \mathcal{W}(\pi\subscriptp, \psi)$, we define $$\langle W_1, W_2 \rangle = \int_{F^\times} W_1(a(t)) \overline{W_2(a(t))}d^\times t.$$
\end{definition}

\begin{lemma}\label{normwhittaker}Let $n\ge 1$. We  have \begin{enumerate}

\item $\langle W_\pi, W_\pi \rangle = \langle W_\pi^*, W_\pi^*\rangle.$

\item  $1 \le \langle W_\pi, W_\pi\rangle \le 2.$
  \item If $L(s, \pi)=1$, then  $\langle W_\pi, W_\pi\rangle = 1.$
 \end{enumerate}
\end{lemma}
\begin{proof}All the parts are immediate from Lemma~\ref{whitdiagonal}.
\end{proof}

\begin{lemma}\label{conjlemma}For all $g \in G$, we have $W_\pi^*(g) = c_\pi W_\pi(g \mat{0}{1}{\varpi^n}{0}),$ where $|c_\pi|=1$.
\end{lemma}
\begin{proof}
It is clear that the function $W'$ defined via $W'(g) = W_\pi(g \mat{0}{1}{\varpi^n}{0})$ lies in $\mathcal{W}(\pi,\psi)$ and is right invariant by $K_2(\varpi^n).$ It follows that $W_\pi^*$ is a multiple of $W'$. Moreover, by the $G$-invariance of
$\langle , \rangle$
on $\mathcal{W}(\pi,\psi)$, we have $\langle W', W' \rangle = \langle W_\pi, W_\pi \rangle$. The result now follows from the first part of Lemma~\ref{normwhittaker}.
\end{proof}

\begin{lemma}\label{cpicompute2}Let $t \in \Z$, $0\le k \le n$, $v \in \OF^\times$. Then $$
W_\pi^*(g_{t,k,v}) = c_\pi \omega_\pi(\varpi^{n-k}v)\psi(-\varpi^{t+k}v^{-1}) W_\pi(g_{t+2k-n,n-k,-v}).$$ In particular, $|W_\pi^*(g_{t,k,v})| = |W_\pi(g_{t+2k-n,n-k,-v})|$.
\end{lemma}
\begin{proof}This follows from the equation
$$g_{t,k,v} \mat{0}{1}{\varpi^n}{0} = n(-\varpi^{t+k}v^{-1})z(\varpi^{n-k}v)g_{t+2k-n,n-k,-v} \mat{1}{0}{0}{v^{-2}}.$$
\end{proof}

\begin{remark}Combined with the previous remark, this shows that $$|W_\pi(g_{-n,0,-1})|=|W_\pi^*(g_{-n,0,-1})|=1.$$

\end{remark}

\begin{lemma}\label{supportlemma}Let  $0\le k \le n$, $v \in \OF^\times$. Then for any integer $t<-k-n$, we have $W_\pi^*(g_{t,k,v}) = W_\pi(g_{t,k,v}) = 0$.
\end{lemma}
\begin{proof} If $t<-k-n$, we can write down elements $n(x) \in N$ , $k \in K_2(\p^n)$, such that $\psi(x) \neq 1$ and $g_{t,k,v} = n(x)g_{t,k,v}k$. We omit the elementary details.
\end{proof}

\subsection{The local functional equation}The local functional equation, due to Jacquet and Langlands~\cite{MR0401654}, will be key for our approach.

\begin{theorem}[Jacquet-Langlands]For each $W' \in
\mathcal{W}(\pi\subscriptp,\psi)$, each character $\mu$ of
$F^\times$, and each
complex number $s \in \mathbb{C}$,
the local zeta
integral
\[
Z(W',s,\mu) = \int_{F^\times} W'(a(y)) \mu(y) |y|^{s-\frac12} \, d^\times y
\]
satisfies
\begin{equation}\label{gl2-loc-func-eqn}
  \frac{Z(W',s,\mu)}{L(s,\pi\subscriptp \twist \mu)}\eps(s,\pi\subscriptp \twist \mu)  = \frac{Z(w \cdot W',1-s,\mu^{-1}\omega_\pi^{-1})}{L(1-s, \pi\subscriptp \twist \mu^{-1}\omega_\pi^{-1})},
\end{equation}
and moreover, each side is a polynomial in the variables $q^s, q^{-s}$.
\end{theorem}

Now, for each $t \in \Z$, $0\le k \le n$,  consider the function on $\OF^\times$ given by $v \mapsto W_\pi(g_{t,k,v})$. It is easy to check that this function depends only on $\OF^\times / (1 + \p^k)$. So, by Fourier inversion, it follows that for each character $\mu \in \tilde{X}(k)$, there exists a complex number $c_{t,k}(\mu)$ such that \begin{equation}\label{Wpiexpand}W_\pi(g_{t,k,v}) = \sum_{\mu \in \tilde{X}(k)} c_{t,k}(\mu) \mu(v).\end{equation} Similarly, we may define a complex number $c_{t,k}^*(\mu)$ such that \begin{equation}\label{Wpistexpand}W_\pi^*(g_{t,k,v}) = \sum_{\mu \in \tilde{X}(k)} c_{t,k}^*(\mu) \mu(v).\end{equation}

Moreover, by the orthogonality of characters, we have:

\begin{equation}\label{ctkdef} c_{t,k}(\mu) = \int_{\OF^\times}W_\pi(g_{t,k,v})\mu^{-1}(v) d^\times v \end{equation}

\begin{equation}\label{ctkstdef} c_{t,k}^*(\mu) = \int_{\OF^\times}W_\pi^*(g_{t,k,v})\mu^{-1}(v) d^\times v \end{equation}

Note that $c_{t,k}(\mu) = c_{t,k}^*(\mu) = 0$ if $t<-k-n$; this follows from Lemma~\ref{supportlemma}.

\begin{remark}We define $c_{t,k}(\mu)$ and $c_{t,k}^*(\mu)$  for all characters $\mu \in \tilde{X}$ using the equations~\eqref{ctkdef} and \eqref{ctkstdef}. It is immediate that $c_{t,k}(\mu) = c_{t,k}^*(\mu) = 0$ whenever $a(\mu)>k$.
\end{remark}

\subsection{The basic identity}
Now, we apply~\eqref{gl2-loc-func-eqn} with $W' = w.n(\varpi^{-k}).W_\pi$. Let $\mu \in \tilde{X}(k)$, $\omega_\pi \in \tilde{X}$. We have \begin{align*}Z(W',s,\mu) &= \int_{F^\times} W_\pi(a(y)w.n(\varpi^{-k})) \mu(y) |y|^{s-\frac12} \, d^\times y \\ &= \sum_{t=-\infty}^\infty q^{t(\frac12-s)} \int_{\OF^\times} W_\pi(a(\varpi^tv)w.n(\varpi^{-k})) \mu(v) \, d^\times v \\&= \sum_{t=-\infty}^\infty q^{t(\frac12-s)} \int_{\OF^\times} W_\pi(a(\varpi^t)w.n(\varpi^{-k}v^{-1})) \mu(v) \, d^\times v \\&= \sum_{t=-\infty}^\infty q^{t(\frac12-s)} \int_{\OF^\times} W_\pi(g_{t,k,v}) \mu^{-1}(v) \, d^\times v
\\&=  \sum_{t=-\infty}^\infty q^{t(\frac12-s)} c_{t,k}(\mu) \\&=  \sum_{t=-k-n}^\infty q^{t(\frac12-s)} c_{t,k}(\mu)
\end{align*}

On the other hand \begin{align*}Z(w \cdot W',1-s,\mu^{-1}\omega_\pi^{-1}) &= \int_{F^\times} W_\pi(-a(y)n(\varpi^{-k})) \omega_\pi^{-1}(y)\mu^{-1}(y) |y|^{\frac12-s} \, d^\times y \\ &= \omega_\pi(-1)\sum_{a=-\infty}^\infty \int_{\OF^\times} W_\pi(a(\varpi^av)n(\varpi^{-k})) \omega_\pi^{-1}(\varpi^av)\mu^{-1}(v) q^{-a(\frac12-s)} \, d^\times v \\&= \omega_\pi(-1)\sum_{a=0}^\infty q^{-a(\frac12-s)}\int_{\OF^\times} W_\pi(a(\varpi^av))\psi(\varpi^{a-k}v) \omega_\pi^{-1}(\varpi^av)\mu^{-1}(v) \, d^\times v \\&= \omega_\pi(-1)\sum_{a=0}^\infty q^{-a(\frac12-s)}W_\pi(a(\varpi^a))\int_{\OF^\times} \psi(\varpi^{a-k}v) \mu^{-1}(v) \, d^\times v.
\end{align*}

Now, using~\eqref{gl2-loc-func-eqn} and the formula $\eps(s,\pi\subscriptp \twist \mu) = \eps(1/2,\pi\subscriptp \twist \mu)q^{a(\mu\pi)(\frac12-s)}$ we immediately arrive at an identity relating the various $c_{t,k}(\mu)$ with terms of $L$ and $\eps$-factors. Moreover, one can repeat the whole sequence of steps with $W_\pi^*$ instead of $W_\pi$ to get a dual identity. We collect both identities together in the following Proposition, omitting the easy proof.

\begin{proposition}[The basic identity]\label{basicformula} Let $0\le k \le n$. For each integer $r$ and each character $\mu \in \tilde{X}(k)$, let the quantities $c_{t,k}(\mu)$, $c_{t,k}^*(\mu)$ be as in the previous subsection.
Suppose that $\omega_\pi \in \tilde{X}$. Then the following identities (of polynomials in $q^s$, $q^{-s}$) hold.

\begin{equation}\label{basicid}\begin{split}&\varepsilon(\frac12, \mu\pi)\left(\sum_{t=-k-n}^\infty q^{(t+a(\mu\pi))(\frac12-s)} c_{t,k}(\mu)\right) L(s, \mu \pi)^{-1} \\&= \omega_\pi(-1)\left(\sum_{a=0}^\infty W_\pi(a(\varpi^a)) q^{-a(\frac12-s)}G(\varpi^{a-k}, \mu^{-1})\right) L(1-s, \mu^{-1}\omega_\pi^{-1}\pi)^{-1}.\end{split}\end{equation}

\begin{equation}\label{basicideqst}\begin{split}&\varepsilon(\frac12, \mu\omega_\pi^{-1}\pi)\left(\sum_{t=-k-n}^\infty q^{(t+a(\mu\omega_\pi^{-1}\pi))(\frac12-s)} c^*_{t,k}(\mu)\right) L(s, \mu \omega_\pi^{-1} \pi)^{-1} \\&= \omega_\pi(-1)\left(\sum_{a=0}^\infty W_\pi(a(\varpi^a))q^{-a(\frac12-s)}G(\varpi^{a-k}, \mu^{-1})\right) L(1-s, \mu^{-1}\pi)^{-1} .\end{split}\end{equation}
\end{proposition}

\begin{remark}In~\eqref{basicideqst}, we have used the fact that  $W_\pi^*(a(\varpi^a)) = W_\pi(a(\varpi^a)).$
\end{remark}

\begin{remark}Implicit in the statement of the above Proposition is the fact that the total quantity on either side of each $=$ sign  is a \emph{polynomial} (and not  a power/Laurent series) in the variables $q^s$, $q^{-s}$.
\end{remark}

\begin{corollary}\label{prevtil}Suppose that $\omega_\pi \in \tilde{X}$. We have $W_{\tilde{\pi}}(g_{t,k,v}) =  W_\pi^*(g_{t,k,v}).$
\end{corollary}
\begin{proof}Let $\tilde{c}_{t,k}$  be the corresponding coefficient for $\tilde{\pi}$. If, in~\eqref{basicid}, we replace $\pi$ with $\tilde{\pi}$, then we get exactly~\eqref{basicideqst}, except that instead of $c^*_{t,k}(\mu)$ we get $\tilde{c}_{t,k}(\mu)$. Hence $c^*_{t,k}(\mu)=\tilde{c}_{t,k}(\mu)$ and the result follows.
\end{proof}
\begin{corollary}\label{prevtil2}Suppose that $\omega_\pi \in \tilde{X}$. We have $W_{\tilde{\pi}}(g) =  \omega_{\pi}^{-1}(\det(g))W_\pi^*(g)$.
\end{corollary}
\begin{proof}Using Lemma~\ref{disjoint}, write $g = z(u)ng_{t,k,v}g_0$ where $z(u)\in Z$, $n = n(x) \in N$, $g_0 \in K_1(\p^n)$. Note that $\det(g) = u^2 \det(g_0)\varpi^t$. We have $W_{\tilde{\pi}}(g) = W_{\tilde{\pi}}(g_{t,k,v}) \psi(x) \omega_{\pi}^{-1}(u)$. On the other hand, writing
$g = z(\det(g_0)u)ng_{t,k,v}(z(\det(g_0))^{-1}g_0)$, we see that $W_\pi^*(g) = \psi(x) \omega_\pi(\det(g_0)u)W_\pi^*(g_{t,k,v}).$ Now the result follows from Corollary~\ref{prevtil}.
\end{proof}

The next proposition will be of key importance for us as it will allow us to restrict ourselves to the case $0 \le k\le n/2$ in all future calculations.
\begin{proposition}\label{cpiformula}Suppose that $\omega_\pi \in \tilde{X}$. Then, for $t \in \Z$, $0\le k \le n$, $v \in \OF^\times$, we have $$
W_{\tilde{\pi}}(g_{t,k,v}) = \eps(1/2,  \pi) \omega_\pi(v)\psi(-\varpi^{t+k}v^{-1}) W_\pi(g_{t+2k-n,n-k,-v}).$$
\end{proposition}
\begin{proof}It suffices to prove that the constant $c_\pi$ appearing in Lemma \ref{cpicompute2} equals $\eps(1/2, \pi)$. Using Remark~\ref{cpicompute1} and Lemma \ref{cpicompute2}, we see that $W^*_{\pi}(g_{-n,0,-1}) = c_\pi$. On the other hand, we have by definition $W^*_{\pi}(g_{-n,0,-1}) = c^*_{-n,0}(1)$. Comparing the constant coefficients of both sides of \eqref{basicideqst}, we see that
$c^*_{-n,0}(1)\eps(1/2, \omega_\pi^{-1} \pi)  = \omega_\pi(-1)$ which implies that $$ c_\pi = c^*_{-n,0}(1) = \frac{\omega_\pi(-1)}{\eps(1/2, \omega_\pi^{-1} \pi)} = \eps(1/2, \pi).$$\end{proof}

\begin{remark}In the special case $\omega_\pi = 1$ the above Proposition is nothing but the action of the Atkin-Lehner operator on the Whittaker model.
\end{remark}

\subsection{The basic identity in the supercuspidal case}

Even though we will not need it for this paper, it is instructive to see what the basic identity gives when $\pi$ is supercuspidal.

\begin{proposition}\label{prop:whitsupercusp}Let $\pi$ be an
irreducible unitarizable supercuspidal representation of $G$ and assume $\omega_\pi \in \tilde{X}$. Then the following hold.
\begin{itemize}
\item $W_\pi(g_{t,0,v}) = \begin{cases} 0 & \text{ if } t \neq -n\\ \eps(1/2, \omega_{\pi}^{-1}\pi) &\text{ if }t= -n.\end{cases}$

\item Let $0<k \le n$. Then, \begin{equation}\label{supercuspsum}W_\pi(g_{t,k,v}) = G(\varpi^{-k}, 1)\varepsilon(1/2, \omega_\pi^{-1} \pi)\delta_{
    t,-n} +  \zeta_F(1) \  q^{-\frac{k}2} \sum_{\substack{\mu \in \tilde{X} \\ a(\mu) =k\\  \ a(\mu \pi) = -t}} \varepsilon(1/2, \mu) \ \varepsilon(1/2, \mu^{-1} \omega_\pi^{-1}\pi) \mu(v). \end{equation}

\end{itemize}

\end{proposition}
\begin{proof}Since $\pi$ is supercuspidal, we have  $L(s, \mu \pi)= L(1-s, \mu^{-1}\omega_\pi^{-1}\pi)=1.$ Now applying the basic identity, and using Lemma~\ref{whitdiagonal}, we get the desired result.
\end{proof}

\begin{remark}Proposition~\ref{prop:whitsupercusp} gives an exact formula for any value of the Whittaker newform as a sum of $\ll q^k$ terms. Moreover, if $k>n/2$, we may use Proposition~\ref{cpiformula} to work with $n-k$ instead of $k$. Thus we can always express any value of the Whittaker newform as a sum of $\ll q^{k_n}$ terms, each of absolute value $\asymp q^{-k_n/2}$, where we denote $k_n = \min(k, n-k)$.
\end{remark}

\subsection{The upper bound}
Recall that we are assuming $n \ge 1$. We now prove the upper bound of Theorem~\ref{t:main}, noting that the proof for $n\le 1$ follows from Lemma~\ref{trivboundn0}.

\begin{definition}For each $t\in \Z$, and each $0 \le k \le n$, define $$\lambda_{\pi,t,k} = \left(\int_{v \in \OF^\times}|W_\pi(g_{t,k,v})|^2 d^\times v\right)^{1/2}.$$
\end{definition}

\begin{lemma}We have $\lambda_{\pi,t,k} = \lambda_{\tilde{\pi}, t+2k-n, n-k}$.
\end{lemma}
\begin{proof} This is immediate from Proposition~\ref{cpiformula}.
\end{proof}

\begin{lemma}  We have $1\le \sum_{t =-\infty}^\infty \lambda_{\pi,t,k}^2 \le 2.$  If $L(s, \pi)=1$, then $\sum_{t =-\infty}^\infty \lambda_{\pi,t,k}^2 = 1.$

\end{lemma}
\begin{proof}Let $W' = \pi(g_{0,k,1})W_\pi$. Then, by definition, $\langle W', W' \rangle = \sum_{r =-\infty}^\infty \lambda_{\pi,t,k}^2.$ On the other hand, by the invariance of the inner product, we have $\langle W', W'\rangle = \langle W_\pi, W_\pi\rangle.$ Now the result follows from Lemma~\ref{normwhittaker}.
\end{proof}

\begin{corollary} We have $h(\pi) \le \sqrt{2} \ q^{\frac12\lfloor \frac{n}{2} \rfloor}.$

\end{corollary}
\begin{proof}We may assume $n\ge 1$. Let $k_0$ be an integer such that $h(\pi) = W_\pi(g_{r_0,k_0,v_0})$ for some $r_0$, $v_0$. By replacing $\pi$ with $\tilde{\pi}$  if necessary, we may assume $k_0 \le \lfloor n/2 \rfloor$. Let $c_{k_0} = \vol(1 + \p^{k_0}) \ge q^{-\lfloor \frac{n}{2} \rfloor}$. Note that the function $v \mapsto W_{\pi}(r_0, k_0, v)$ is $(1 + \p^{k_0})$ invariant. Hence $c_{k_0} W_\pi(g_{r_0,k_0,v_0})^2 \le \lambda_{\pi,r_0,k_0}^2 \le 2$, leading to $h(\pi) \le \sqrt{2} \ q^{\frac12\lfloor \frac{n}{2} \rfloor}.$

\end{proof}

Now that we have proved the upper bound, we only need to show that $$h(\pi) \gg  q^{\frac12 \lfloor \frac{3m}{2} \rfloor - \frac{n}2}.$$ It suffices to consider the case $m > \frac{2n}{3}$ as otherwise the right side is less than or equal to 1. In this case,  $\pi$ must be an induced representation, as shown by the following lemma.

\begin{lemma}Suppose that $m = a(\omega_\pi) >n/2$.  Then $\pi = \chi_1 \boxplus \chi_2$, where $\chi_1$, $\chi_2$ are unitary characters with $a(\chi_1) = m$,  $ a(\chi_2) = n-m$.
\end{lemma}
\begin{proof}$\pi$ cannot be a twist  $\chi \rm{St}$ of the Steinberg representation, as these have $m=a(\chi^2)$, $n = 2a(\chi)$; hence $m \le n/2$. Similarly $\pi$ cannot be a supercuspidal representation because a result of Tunnell~\cite{tunnell78} tells us that $m=a(\omega_\pi) \le \frac12 a(\pi) =\frac{n}{2}.$ So $\pi$ must be isomorphic to $\chi_1 \boxplus \chi_2$, where $\chi_1$, $\chi_2$ are unitary characters with $a(\chi_1) = a_1$,  $ a(\chi_2) = a_2$. We may assume without loss of generality that $a_1 \ge a_2$. Then $n=a_1 + a_2$, $m=a(\chi_1 \chi_2)$. Since $m>n/2$, we deduce immediately that $m=a_1$, $a_2 = n-m$.
\end{proof}

\subsection{Some useful facts}
We prove some elementary facts on $\GL(1)$ $\eps$-factors that will be useful to us in the next subsection. These facts can probably be found elsewhere, but we give proofs here for completeness.

\begin{lemma}\label{twistepsilonlemma}Let $\chi \in \tilde{X}$ be a character with $a(\chi)=r \ge 1$. Let $r_0= \lfloor \frac{r}{2}\rfloor$.
\begin{enumerate}
\item There exists $v \in \OF^\times$ satisfying  $$\chi(1 + \varpi^{r-r_0}u) = \psi(v^{-1} \varpi^{-r_0}u)$$ for all $u \in \OF$.
\item For all $\mu \in \tilde{X}$ satisfying $a(\mu) \le r_0$, we have $$\varepsilon(1/2, \mu^{-1}\chi^{-1})\mu(-v) = \varepsilon(1/2, \chi^{-1}).$$
\end{enumerate}

\end{lemma}

\begin{proof} The first fact is immediate from the fact that $\psi'(x):= \chi(1 + \varpi^{r-r_0}x)$ is an additive character on $\OF$. So, there must exist $y \in F$ such that $\psi'(x) = \psi(xy)$ for all $x \in \OF$. Comparing conductors, we see that  $v(y) = -r_0$. So we may put $y = v^{-1} \varpi^{-r_0}$ for some $v \in \OF^\times$.

The second fact follows from the following calculation:

\begin{align*}
&\zeta_F(1) q^{-r/2}\varepsilon(1/2, \mu^{-1}\chi^{-1})\mu(-v) =G(\varpi^{-r}, \mu \chi) \mu(-v)\\&=\int_{y\in \OF^\times}\psi(\varpi^{-r} y) \chi(y) \mu(-vy) = \sum_{a\in \OF^\times/(1+\p^{r-r_0})} \chi(a) \mu(a) \int_{y \in (1+\p^{r-r_0})}\psi(\varpi^{-r} ay) \chi(y) \mu(-vy) d^\times y \\ &=q^{r_0-r}\sum_{a\in \OF^\times/(1+\p^{r-r_0})} \chi(a) \mu(a) \psi(\varpi^{-r}a)\int_{y \in \OF}\psi(\varpi^{-r_0} ay) \chi(1 + \varpi^{r-r_0}y) \mu(-v(1 + \varpi^{r-r_0}y))  dy\\&= q^{r_0-r}\sum_{a\in \OF^\times/(1+\p^{r-r_0})} \mu(-v) \chi(a) \mu(a) \psi(\varpi^{-r}a)\int_{y \in \OF}\psi(y(\varpi^{-r_0} a + v^{-1} \varpi^{-r_0})) \\&= q^{r_0-r}\sum_{a\in -v^{-1}(1+\p^{r_0})/(1+\p^{r-r_0})} \chi(a)  \psi(-\varpi^{-r}a),
\end{align*}

which does not depend on $\mu$. So taking $\mu=1$, we deduce the desired result.
\end{proof}

\begin{lemma}\label{twistepsilonlemma2}Let $0<r'<r$ be integers, and let $\chi \in \tilde{X}$ be a character with $a(\chi)=r'$.
Then $$\big|\sum_{\substack{\mu \in \tilde{X} \\ a(\mu) = r}} \eps(1/2, \mu^{-1}) \eps(1/2, \mu \chi) \mu(v)\big| = \begin{cases}\zeta_F(1)^{-1} q^{\frac{r-r'}{2}}  &\text{ if } v \in -1 + \p^{r-r'}\OF^\times \\ 0 & \text{ otherwise.} \end{cases}$$

\end{lemma}
\begin{proof}
This follows from the following calculation:

\begin{align*}
&\zeta_F(1)^2 q^{-r}\big|\sum_{\substack{\mu \in \tilde{X} \\ a(\mu) = r}} \eps(1/2, \mu^{-1}) \eps(1/2, \mu \chi) \mu(v)\big|\\&=\big| \sum_{\substack{\mu \in \tilde{X} \\ a(\mu) = r}}G(\varpi^{-r}, \mu)G(\varpi^{-r}, \mu^{-1}\chi^{-1})\mu(v)\big|\\&=\big| \sum_{\substack{\mu \in \tilde{X} \\ a(\mu) \le r}}G(\varpi^{-r}, \mu)G(\varpi^{-r}, \mu^{-1}\chi^{-1})\mu(v)\big|\\&=\big|\int_{y_1, y_2\in \OF^\times}\psi(\varpi^{-r} (y_1 + y_2)) \chi(y_2^{-1}) \sum_{\substack{\mu \in \tilde{X} \\ a(\mu) \le r}}\mu(y_1y_2^{-1}v)\big|\\&=\big|\int_{y_1 \in \OF^\times}\int_{u \in 1 + \p^r}\psi(\varpi^{-r} (y_1 + y_1vu)) \chi(y_1^{-1}v^{-1}) \sum_{\substack{\mu \in \tilde{X} \\ a(\mu) \le r}}1\big| \\& =\big|\int_{y_1 \in \OF^\times}\psi(\varpi^{-r} (y_1 + y_1v)) \chi(y_1^{-1}) \big| \\& =\big|G(\varpi^{-r}(1+v), \chi^{-1}) \big|
\end{align*}
The result now follows from~\eqref{formulagauussum}.
\end{proof}

\subsection{The lower bound}

\emph{For the rest of this section, we only consider the case $\pi = \chi_1 \boxplus \chi_2$, where $\chi_1$, $\chi_2$ are unitary characters with $a_1 =a(\chi_1)$,  $a_2 = n-a_1 = a(\chi_2)$, and $\frac{a_1}{2} >a_2 \ge 0.$} We need to prove that  $$h(\pi) \gg  q^{\frac12 \lfloor \frac{a_1}{2} \rfloor - \frac{a_2}2}.$$ We will do this by exhibiting a specific triple $(t,k,v)$ for which $W_\pi(g_{t,k,v}) \asymp q^{\frac12 \lfloor \frac{a_1}{2} \rfloor - \frac{a_2}2}.$

\subsubsection*{The case $a_2 = 0$}
We first consider the case when $a_2=0$. This case was already done by Templier, but for completeness we give a proof by our method as well. Let $n_0 =\lfloor \frac{n}{2} \rfloor$. We prove:

\begin{proposition}\label{propa20}Let $v_0 \in \OF^\times$ satisfy $$\chi_1(1 + \varpi^{n-n_0}u) = \psi(v_0^{-1} \varpi^{-n_0}u)$$ for all $u \in \OF$. Then $W_\pi(g_{-n_0 -n,n_0,v_0}) \asymp q^{n_0/2}.$
\end{proposition}

\begin{proof}
Let $0 \le k \le n/2$. If $\mu \in \tilde{X}(k)$ is a character such that $\mu \neq 1$, then the basic identity~\eqref{basicid} becomes \begin{equation}\varepsilon(\frac12, \mu\pi)\left(\sum_{t=-k-n}^\infty q^{(t+n+a(\mu))(\frac12-s)} c_{t,k}(\mu)\right)= \omega_\pi(-1)\left(\sum_{a=0}^\infty q^{-a(1-s)} G(\varpi^{a-k},\mu^{-1})\right) .\end{equation}

Using~\eqref{formulagauussum}, this tells us that \begin{equation}\label{crka2zero}c_{t,k}(\mu) = \zeta_F(1) \chi_2(\varpi^{-n})\cdot\begin{cases} q^{-\frac{k}{2}}\mu(-1)\varepsilon(\frac12, \mu^{-1}\chi_1^{-1}) & \text{ if } t= - k-n,\\ 0 & \text{ otherwise. } \end{cases}\end{equation}

On the other hand, if $\mu=1$, then the basic identity becomes \begin{equation}\begin{split}&\varepsilon(\frac12, \chi_1)\left(\sum_{t=-k-n}^\infty q^{(t+n)(\frac12-s)} c_{t,k}(1)\right)(1- \chi_2(\varpi)q^{-s})\\&= \omega_\pi(-1)\left(\sum_{a=0}^\infty q^{-a(1-s)} G(\varpi^{a-k},1)\right) (1- \chi_2(\varpi^{-1})q^{s-1}) .\end{split}\end{equation} from which it is immediate that \begin{equation}\label{crka2zeromu1} c_{-n-k,k}(1) \asymp 1 \end{equation}

We now note that \begin{align*}|W_\pi(g_{-n_0 -n,n_0,v_0}) |&= |\sum_{\mu \in \tilde{X}(n_0)}c_{-n_0 -n, n_0}(\mu) \mu(v_0) |\\&\asymp 1 + \bigg|\sum_{\substack{\mu \in \tilde{X}(n_0)\\ \mu \neq 1}}c_{-n_0 -n,n_0}(\mu) \mu(v_0)\bigg| \\&\asymp 1 + q^{-\frac{n_0}{2}}\bigg| \sum_{\substack{\mu \in \tilde{X}(n_0)\\ \mu \neq 1}}\varepsilon(\frac12, \mu^{-1}\chi_1^{-1})\mu(-v_0)\bigg| \\&= 1 + q^{-\frac{n_0}{2}}\bigg| \sum_{\substack{\mu \in \tilde{X}(n_0)\\ \mu \neq 1}}1\bigg|
\end{align*}
 where the last step uses   Lemma~\ref{twistepsilonlemma}. We immediately conclude that $|W_\pi(g_{-n_0 -n,n_0,v_0}) | \asymp q^{n_0/2}$.
\end{proof}

\subsubsection*{The case $\frac{n}{2}> \frac{a_1}{2} \ge \lfloor \frac{a_1}{2} \rfloor > a_2>0$}

In this case, we prove:
\begin{proposition}Let $v_0 \in \OF^\times$ satisfy $$\chi_1(1 + \varpi^{a_1-\lfloor \frac{a_1}{2}\rfloor}u) = \psi(v_0^{-1} \varpi^{-\lfloor \frac{a_1}{2}\rfloor}u)$$ for all $u \in \OF$, and let $w_0 \in v_0(1+\varpi^{\lfloor \frac{a_1}{2}\rfloor - a_2}\OF^\times)$. Then $W_\pi(g_{-\lfloor \frac{3a_1}2 \rfloor,\lfloor \frac{a_1}2 \rfloor,w_0}) \asymp q^{\frac12 \lfloor \frac{a_1}{2} \rfloor - \frac{a_2}2}.$

\end{proposition}
\begin{proof}
 Let $0 \le k \le n/2$, and let $\mu \in \tilde{X}(k)$ be a character such that $\mu| \OF^\times \neq \chi_2^{-1} | \OF^\times$. In this case the basic identity \eqref{basicid} reduces to \begin{equation}\varepsilon(\frac12, \mu\chi_1)\varepsilon(\frac12, \mu\chi_2)\left(\sum_{t=-k-n}^\infty q^{(t+ a_1 + a(\mu\chi_2))(\frac12-s)} c_{t,k}(\mu)\right) = \omega_\pi(-1)G(\varpi^{-k}, \mu^{-1})\end{equation}

This tells us that $$c_{t,k}(\mu) = \begin{cases} \frac{\omega_\pi^{-1}(-1) G(\varpi^{-k}, \mu^{-1})}{\varepsilon(\frac12, \mu\chi_1)\varepsilon(\frac12, \mu\chi_2)} & \text{ if } t= -a_1 - a(\mu\chi_2), \\ 0 & \text{ otherwise. } \end{cases}$$

Next, suppose $\mu \in \tilde{X}(k)$ satisfies $\mu| \OF^\times = \chi_2^{-1}| \OF^\times$; this implies $a(\mu) = a_2$. In this case,  \eqref{basicid} reduces to \begin{equation}\varepsilon(\frac12, \mu\chi_1)\left(\sum_{t=-k-n}^\infty q^{(t+ a_1 + a(\mu\chi_2))(\frac12-s)} c_{t,k}(\mu)\right) (1- \chi_2(\varpi)q^{-s}) = \omega_\pi(-1)G(\varpi^{-k}, \mu^{-1})(1- \chi_2(\varpi^{-1})q^{s-1})\end{equation}

This implies that $c_{-\lfloor \frac{3a_1}2 \rfloor,\lfloor \frac{a_1}2 \rfloor,}(\mu) = 0$ in this case.

So we see that

\begin{align*}|W_\pi(g_{-\lfloor \frac{3a_1}2 \rfloor,\lfloor \frac{a_1}2 \rfloor,w_0}) |&= \bigg|\sum_{\substack{\mu \in \tilde{X}\\ a(\mu) =\lfloor \frac{a_1}2 \rfloor}}c_{-\lfloor \frac{3a_1}2 \rfloor,\lfloor \frac{a_1}2 \rfloor,}(\mu) \mu(w_0)\bigg| \\&=  \bigg|\sum_{\substack{\mu \in \tilde{X}\\ a(\mu) =\lfloor \frac{a_1}2 \rfloor}}q^{-\frac12\lfloor \frac{a_1}{2} \rfloor}\eps(1/2, \mu^{-1} \chi_1^{-1}) \eps(1/2, \mu^{-1} \chi_2^{-1}) \eps(1/2, \mu) \mu(w_0)\bigg|\\&= \bigg|\sum_{\substack{\mu \in \tilde{X}\\ a(\mu) =\lfloor \frac{a_1}2 \rfloor}}q^{-\frac12\lfloor \frac{a_1}{2} \rfloor}\eps(1/2, \mu^{-1} \chi_2^{-1}) \eps(1/2, \mu) \mu(-v_0^{-1}w_0)\bigg| \\ &\asymp q^{\frac12 \lfloor \frac{a_1}{2} \rfloor - \frac{a_2}2}.
\end{align*}
 where we have used Lemma~\ref{twistepsilonlemma} and   Lemma~\ref{twistepsilonlemma2}.
\end{proof}

\section{The global application}
\subsection{Some notations}\label{s:notfin}Henceforth, we let $F$ denote a \emph{number field}. Let $\mathfrak
{d}$ be the different of $F$ and $\Delta = N_{F/\Q}(\mathfrak{d})$ be the discriminant. For any place $v$ of $F$, we will use the notation $X_v$ for each \emph{local object} $X$ introduced in the previous section. The corresponding global objects will be denoted without the subscript $v$. Thus, we will talk about objects like $F_v, \OF_v, \p_v, \varpi_v, q_v, G_v$ etc., which (at least when $v$ is non-archimedean) were considered in the previous section.  We let $\A = \A_F$ denote the ring of adeles of $F$. We  let $\a $ denote the set of archimedean places and $\f$ denote the set of non-archimedean places. For each $v \in \f$, let $d_v$ be the unique non-negative integer such that $\mathfrak
{d}_v = \p_v^{d_v}$. We let $\psi$ denote the standard non-trivial additive character of $F \bs \A$ obtained by composing the map $\mathrm{Tr}_{\A_F/\A_\Q}$ with the unique additive character on $\A_{\Q}$ that is unramified at all finite places and equals  $e^{2 \pi i x}$ at $\R$. Note that for a place $v \in \f$, $\psi_v$ has conductor equal to $\p_v^{-d_v}$.

For each place $v$ of $F$, we define a maximal compact subgroup $K_v$ of $G_v$ as follows: $$K_v = \begin{cases}G(\OF_v) &\text{ if } v \text{ is non-archimedean,} \\
O(2) &\text{ if } v \text{ is archimedean and real,} \\U(2) &\text{ if } v \text{ is archimedean and complex.}
\end{cases}$$
We let $K = \prod_vK_v$ be the corresponding subgroup of $G(\A)$.

We normalize measures at the archimedean places $v \in \a$ as follows. We take the measure $dx$ on $F_v$ to be the usual Lebesgue measure. We choose the measure on $F_v^\times$ to equal $\zeta_{v}(1)\frac{dx}{|x|}.$ These measures transport to measures on $N_v$, $A_v$ and $Z_v$ for each $v \in \a$. This gives us a left Haar measure on $B_v$. We normalize the measure on $K_v$ to be the probability measure. The Iwasawa decomposition now gives us a left Haar measure on $G_v$.

 We adopt measures on each of our adelic groups by taking the product measure over all places. We give all discrete groups (such as $G(F)$, $N(F)$ etc.) the counting measure and thus obtain a measure on the appropriate quotient groups. In particular, our normalization of measures satisfies the conditions of~\cite[Section 2.1.6]{michel-2009}.

Let $\pi = \otimes_v \pi_v$ be an irreducible, unitary, cuspidal automorphic representation of $G(\A)$ with central character $\omega_\pi = \prod_v \omega_{\pi_v}$. For each $v \in \f$, let $n_v = a(\pi_v)$, $m_v =a(\omega_{\pi_v})$ and put $N = \prod_{v \in \f} q_v^{n_v}$, $M= \prod_{v\in \f} q_v^{m_v}$. We have $N = N_{F/\Q}(\mathfrak{N})$, $M = N_{F/\Q}(\mathfrak{M})$ where $\mathfrak{N}$ is the conductor of $\pi$ and  $\mathfrak{M}$ is the conductor of $\omega_\pi$. We let $\mathfrak{N'}$ denote the ideal of $\OF$ such that  $\mathfrak{N'}_v = \p_v^{\min(n_v, 1)}$ for all $v \in \f$; thus $\mathfrak{N'}$ is the ``squarefree" part of $\mathfrak{N}$.

\subsection{Archimedean lowest-weight vectors}\label{s:notarch}

We recall the various possibilities for $\pi_v$ for each $v \in \a$.

\textbf{Case 1 (Principal series representations of $\GL_2(\R)$):} $F_v = \R$, $\pi_v = \chi_1 \boxplus \chi_2$, where for $i=1, 2$, we have $\chi_i = |t|^{s_i} \sgn(t)^{m_i}$, with $m_i \in \{0,1\}$, $m_1 \ge m_2$, $s_i \in \C$, $t_v=s_1 + s_2 \in i\R$ and $s_v = s_1 - s_2 \in i\R \cup (-1,1)$.
\emph{In this case, we define $k_v = m_1-m_2$}.

We remark that in the special case $s_v =0$, $k_v=1$, the representation $\pi_v$ is also known as the limit of discrete series.

\textbf{Case 2 (Discrete series representations of $\GL_2(\R)$):} $F_v = \R$, $\pi_v$ the unique irreducible subrepresentation of $\chi_1 \boxplus \chi_2$, where for $i=1, 2$, we have $\chi_i = |t|^{s_i} \sgn(t)^{m_i}$, with $m_i \in \{0,1\}$, $m_1 \ge m_2$, $s_i \in \C$, $t_v=s_1 + s_2 \in i\R$, $s_v = s_1 - s_2 \in \Z_{>0}$, $s_v \equiv m_1-m_2+1 \pmod{2}$.
\emph{In this case, we define $k_v = s_v+1$}.

Thus, in the above two cases, $k_v$ is the smallest non-negative integer such that $\pi_v$ contains a vector $\phi_v$ satisfying (for all $\theta \in \R$), \begin{equation}\label{weighteq}\pi_v\left( \mat{\cos(\theta)}{\sin(\theta)}{-\sin(\theta)}{\cos(\theta)} \phi_v \right) = e^{ik_v \theta}\phi_v.\end{equation}
\textbf{Case 3 (Principal series representations of $\GL_2(\C)$):} $F_v = \C$, $\pi_v = \chi_1 \boxplus \chi_2$, where for $i=1, 2$, we have $\chi_i = (z \bar{z})^{s_i} (z (z \bar{z})^{-1/2})^{m_i}$, with $m_i \in \Z$, $m_1 \ge m_2$, $s_i \in \C$, $t_v=s_1 + s_2 \in i\R$ and $s_v = s_1 - s_2 \in i\R \cup (-1,1)$.
\emph{In this case, we define $k_v = m_1-m_2$}.

We note that in this case, $k_v$ is the smallest integer such that the restriction of $\pi_v$ to $K_v$ contains a representation of dimension $k_v +1$.

\medskip

\begin{definition}Let $v \in \a$. We say that a vector $\phi_v \in \pi_v$ is a \emph{lowest weight vector} if

\begin{itemize}

\item When $v$ is real (i.e., we are in Case 1 or Case 2), then the relation~\eqref{weighteq} holds for the particular integer $k_v$ defined above.

\item When  $v$ is complex  (i.e., we are in Case 3), the vector $\phi_v$ is contained in a $k_v +1$ dimensional representation of $K_v$ for the particular integer $k_v$ defined above (this is the lowest-dimensional representation of $K_v$ occurring in $\pi_v$), and furthermore satisfies, for all $\theta \in \R$, the relation \begin{equation}\label{weighteqcmplx}\pi_v\left( \mat{e^{i\theta}}{0}{0}{e^{-i\theta}} \phi_v \right) = e^{ik_v \theta}\phi_v.\end{equation}

\end{itemize}
\end{definition}

It is well-known that the lowest-weight vector is unique up to multiples.

\subsubsection*{The lowest weight vector in the Whittaker model}
It is possible to write down the lowest weight vector explicitly in the Whittaker model. As before, let $v \in \a$. We define a function $W_v$ on $N_vA_v$ as follows.
\begin{itemize}
 \item If    $v \in \a$ is real, then for $x \in \R$, $y \in \R^\times$, $W_v(n(x)a(y)) = e^{2\pi i x} \kappa_v(y)$ where $$\kappa_v(y) = \begin{cases}\sgn(y)^{m_1}|y|^{t_v/2} |y|^{1/2} K_{s_v/2}(2 \pi |y|) &\text{ in Case 1, if } k_v=0,\\|y|^{t_v/2} |y| \left(K_{(s_v-1)/2}(2 \pi |y|)  +\sgn(y)K_{(s_v+1)/2}(2 \pi |y|) \right)&\text{ in Case 1, if } k_v=1, \\
    |y|^{t_v/2} y^{k_v/2}e^{-2 \pi y} (1 + \sgn(y)) &\text{ in Case 2.}\end{cases}$$

 \item If    $v \in \a$ is complex, then for $x \in \C$, $y \in \C^\times$, $W_v(n(x)a(y)) = e^{2\pi i (x+\bar{x})} \kappa_v(y)$ where \begin{equation}\label{formulagl2c}\kappa_v(y) =(y/\bar{y})^{m_1}|y|^{t_v}|y|^{(1+k_v)/2}K_{s_v-k_v/2}(4 \pi |y|).\end{equation}
\end{itemize}

Using the Iwasawa decomposition, we extend $W_v$ to a function on all of $G_v$.\footnote{More precisely, we use~\eqref{weighteq} if $v$ is real, and a combination of ~\eqref{weighteqcmplx} and the degree $k_v$ representation of $SU(2)$ if $v$ is complex.} Then, it can be shown that $W_v$ is (up to multiples) the unique lowest weight vector in $\mathcal{W}(\pi_v, \psi_v)$. A proof of this fact can be found in the paper~\cite{booker-krishna}.

\medskip

We define $$\langle W_v, W_v \rangle = \int_{y \in F_v^\times} |W_v(a(y))|^2 d^\times y.$$ For each place $v \in \a$, define
$$h(\pi_v) = 1 + \frac{\sup_{g \in G_v}|W_v(g)|}{\langle W_v, W_v \rangle^{\frac12}}$$ where $W_v$ is the function defined above. The number $1$ is unimportant, and only exists to ensure that $h(\pi_v)$ is never too small.

It can be checked  using the asymptotics of the Bessel function in the transition range that \begin{equation}\label{hpiformula}h(\pi_v) \asymp \begin{cases} 1 + |s_v|^{1/6} &\text{ in Case 1,} \\ |k_v|^{1/4} = |1 + s_v|^{1/4} &\text{ in Case 2.} \end{cases}\end{equation}

A computation of $h(\pi_v)$ in Case 3, i.e., when $v$ is complex, can probably be done by taking a closer look at the asymptotics of the function $|y|^{(1+k_v)/2}K_{s_v-k_v/2}(4 \pi |y|)$. We do not perform that analysis here.

\begin{remark}\label{remlocanal}Let $C(\pi_v)$  be the local analytic conductor of $\pi_v$. Then there exists an absolute constant $C$ such that $h(\pi_v) \gg C(\pi_v)^C$. This can be seen as follows. By looking at~\eqref{hpiformula}, one can check that $C = 1/12$ works in Cases 1 and 2. For case 3, we can substitute into~\eqref{formulagl2c} $y \asymp |s|$ if $s$ is large relative to $k$ and $|y| \asymp k$ otherwise, to deduce the same result.
\end{remark}
Finally, we define $$h(\pi_\infty) = \prod_{v \in \a}h(\pi_v).$$
\subsection{The main result}We say that an automorphic form $\phi \in \pi$ is a \emph{newform} if $\phi = \otimes_v \phi_v $ is a factorizable vector and $\phi_v \in \pi_v$ satisfy the following conditions:

\begin{enumerate}
\item For each $v \in \f$, we have $\pi(k)\phi_v =\phi_v$ for all $k \in K_1(\p^{n_v})$.

\item For $v \in \a$, $\phi_v$ is a lowest weight vector in $\pi_v$.

\end{enumerate}

It follows that a newform $\phi$ is unique up to multiples. We define $$\|\phi\|_2 = \int_{Z(\A)G(F)\bs G(\A)} |\phi(g)|^2 dg.$$

We now state our main result.

\begin{theorem}\label{t:globalmain}Let the notations be as in Sections~\ref{s:notfin} and \ref{s:notarch}. Let $\phi \in \pi$ be a newform such that $\| \phi \|_2 =1$.  Then $$\sup_{g \in G(\A)}|\phi(g)| \gg_{F,\eps}  h(\pi_\infty)^{1-\eps}N^{-\eps}\prod_{v \in \f}\max(q_v^{\frac12 \lfloor \frac{3m_v}{2} \rfloor - \frac{n_v}2}, 1).$$ In particular, if $\mathfrak{M}$ is the square of an integral ideal and $\mathfrak{N}^2$ divides $\mathfrak{M}^3$,  then $$\sup_{g \in G(\A)}|\phi(g)| \gg_{\pi_\infty,F, \eps } N^{- \eps} \frac{M^{3/4}}{N^{1/2}}.$$

\end{theorem}

\begin{remark}The above Theorem was proved by Templier~\cite{templier-large} in the special case $\mathfrak{M} = \mathfrak{N}$.\footnote{Templier also assumed that $F$ is totally real; however his proof works equally well for general number fields.}
\end{remark}

In view of the above Theorem, it is interesting to speculate on the true size of $\sup_{g \in G(\A)}|\phi(g)|$ relative to the conductor. We propose the following (optimistic) conjecture, which combines Conjecture~\ref{newconj} of the introduction and our expectation that the only obstructions to the sup-norm being as small as possible are the local ones (i.e., the sizes of $h(\pi_v)$).

\begin{conjecture}Let the notations be as in Sections~\ref{s:notfin} and \ref{s:notarch}. Let $\phi \in \pi$ be a newform such that $\| \phi \|_2 =1$. Suppose that $m_v \le \lceil \frac{n_v}{2} \rceil$ for all places $v$. Then $$N^{-\eps} \ll_{\pi_\infty,F, \eps } \sup_{g \in G(\A)}|\phi(g)| \ll_{\pi_\infty,F, \eps }  N^{\eps}.$$
\end{conjecture}

\subsection{The proof of Theorem~\ref{t:globalmain}}Henceforth, we fix a newform $\phi \in \pi$. We define the global Whittaker newform $W_\phi$ on $G(\A)$ in the usual way \begin{equation}\label{globalwhittaker}W_\phi(g) = \int_{F \bs \A} \phi(n(x)g)\psi(-x)dx.\end{equation}

The global Whittaker newform factors as  $W_\phi(g) =c_\phi \prod_v W_v(g)$ where $c_\phi$  is a non-zero complex number independent of $g$, and for each place $v$, the local function $W_v$  is as follows:

\begin{itemize}
\item If $v \in \f$, then $W_v(g) = W_{\pi_v}(a(\varpi_v^{d_v})g)$ where $W_{\pi_v}$ is the local Whittaker newform considered in Section~\ref{sec:local-calculations}. The term $a(\varpi_v^{d_v})$ appears because the conductor of $\psi_v$ is $\p_v^{-d_v}$ (see Remark~\ref{remchoice}).

 \item If    $v \in \a$, then $W_v$ is as in Section~\ref{s:notarch}.
\end{itemize}

Using~\eqref{globalwhittaker}, we get immediately that \begin{align*}\frac{\sup_{g\in G(\A)}|\phi(g)|}{\| \phi \|_2^{1/2}} & \ge \vol(F \bs \A)\frac{\sup_{g\in G(\A)}|W_{\phi}(g)|}{\| \phi \|_2^{1/2}} \\&=  \vol(F \bs \A) \frac{|c_{\phi}|}{\| \phi \|_2^{1/2}} \prod_{v} \sup_{g \in G_v} |W_v(g)| \\&=  \vol(F \bs \A)\frac{|c_{\phi}|}{\| \phi \|_2^{1/2}} h(\pi_\infty) \prod_{v \in \a} \langle W_v, W_v \rangle^{\frac12} \prod_{v \in \f}\max(q_v^{\frac12 \lfloor \frac{3m_v}{2} \rfloor - \frac{n_v}2}, 1), \end{align*}
where in the last step we have used Theorem~\ref{t:main}.

So the theorem will follow once we show that \begin{equation}\label{reqeq}|c_\phi|^2 \gg_{F, \eps} \| \phi \|_2 \left(\prod_{v \in \a} \langle W_v, W_v \rangle \right)^{-1}(N h(\pi_\infty))^{-\eps}.\end{equation}

The theory of the Rankin-Selberg integral representation (see Lemma 2.2.3 of~\cite{michel-2009}) gives us the formula \begin{equation}\label{rankinsel}\frac{\| \phi \|_2 }{L(1, \pi, Ad)} = c_F c_{\phi}^2 \prod_v \frac{\langle W_v, W_v \rangle \ \zeta_v(2) }{L(1, \pi_v, Ad) \ \zeta_v(1) },\end{equation} where $c_F$ is a constant depending only on $F$.

Next, we have the following facts:

\begin{equation}\label{bound1}\frac{\langle W_v, W_v \rangle \ \zeta_v(2) }{L(1, \pi_v, Ad) \ \zeta_v(1) }=1 \text{ for all }v \in \f, \ \pi_v \text{ unramified. \quad (See \cite{michel-2009}.)}\end{equation}  \begin{equation}\label{bound2}L(1, \pi_v, Ad) \asymp 1 \text{ for all } v. \quad\text{ (This is immediate.)}\end{equation} \begin{equation}\label{bound3}\langle W_v, W_v \rangle \asymp 1  \text{ for all }v \in \f, \ \pi_v    \text{ ramified. \quad (See Lemma~\ref{normwhittaker}.)}\end{equation} \begin{equation}\label{bound4}L(1, \pi, Ad) \ll_{F,\eps} (N h(\pi_\infty))^{\eps}\end{equation} The last bound is the usual convexity bound coupled with the observation in Remark~\ref{remlocanal}.

Combining~\eqref{bound1}, ~\eqref{bound2},~\eqref{bound3}, and~\eqref{bound4} with~\eqref{rankinsel}, we immediately derive~\eqref{reqeq}. This completes the proof of Theorem~\ref{t:globalmain}.

\bibliography{refs-que}

% -----------------
\end{document}